\documentclass{article}

\usepackage{amsmath, amssymb, amsthm, amsfonts, bm}
\usepackage{enumerate}

\numberwithin{equation}{section}
\newtheorem{theorem}{Theorem}
\newtheorem{cthm}{Theorem}
\newenvironment{ctheorem}[1]
  {\renewcommand\thecthm{#1}\cthm}
  {\endcthm}
\newtheorem{corollary}{Corollary}
\newtheorem{lemma}{Lemma}[section]
\theoremstyle{remark}
\newtheorem{remark}{Remark}[section]

\newcommand{\Rnum}[1]{\uppercase\expandafter{\romannumeral #1\relax}}
\newcommand{\rnum}[1]{\romannumeral #1\relax}
\newcommand{\mr}[1]{\mathrm{#1}}
\newcommand{\mb}[1]{\mathbb{#1}}

\title{Dyadic-BMO functions, the dyadic Gurov-Reshetnyak conditions on $[0,1]^n$ and rearrangements of functions}
\author{Eleftherios N. Nikolidakis}
\date{\today}

\begin{document}
\maketitle

\begin{abstract}
We introduce the space of dyadic bounded mean oscillation functions $f$ defined on $[0,1]^n$ and study the behavior of the non increasing rearrangement of $f$, as an element of the space $\mr{BMO}\left((0,1]\right)$. We also study the analogous class of functions that satisfy the dyadic Gurov-Reshetnyak condition and look upon their integrability properties.
\end{abstract}

\section{Introduction} \label{sec:1}
It is well known that the space of bounded mean oscillation plays a central role in harmonic analysis and especially in the theory of maximal operators and weights. It is defined by the following way. For an integrable function $f: Q_0\equiv [0,1]^n \to \mb R$ we define the mean oscillation of $f$ on $Q$, where $Q$ is a subcube of $Q_0$ by the following
\begin{equation}  \label{eq:1p1}
\Omega(f,Q) = \frac{1}{|Q|} \int_Q |f(x) - f_Q|\,\mr dx
\end{equation}
where $f_Q = \frac{1}{|Q|} \int_Q f(y)\,\mr dy$, is the integral average of $f$ on $Q$. We will say that $f$ is of bounded mean oscillation on $Q_0$ if the following is satisfied
\[
\|f\|_\star \equiv \sup\left\{ \Omega(f,Q): Q\ \text{is a subcube of}\ Q_0\right\} < +\infty.
\]
We will then write $f\in\mr{BMO}(Q_0)$. \\
We are interested about the behavior of the nonincreasing rearrangement $f^\star$, as an element of $\mr{BMO}((0,1])$, when $f\in\mr{BMO}(Q_0)$. \\
$f^\star$ is defined as the unique equimeasurable to $|f|$ \footnote{in the sense that $\left|\{f^\star>\lambda\}\right| = \left|\{|f|>\lambda\}\right|$, for any $\lambda>0$.}, with domain $(0,1]$, function which also satisfies that it is nonincreasing and left continuous. A discussion about this definition can be seen in \cite{4}. There is also an equivalent definition of $f^\star$ which is given by the following formula:
\begin{equation} \label{eq:1p2}
f^\star(t) = \sup_{\substack{E\subseteq [0,1]^n\\|E| = t\ \ }} \left[\inf_{x\in E} |f(x)|\right],\quad \text{for}\ \ t\in (0,1].
\end{equation}
This can be seen in \cite{8}.

There is also an analogous function, corresponding to $f$, denoted by $f_d$ which is now equimeasurable to $f$, left continuous and nonincreasing. This function now rearranges $f$ and not $|f|$ as $f^\star$ does, so that it is real valued. Also an analogous formula as \eqref{eq:1p2} holds for $f_d$, if we replace the term $|f(x)|$ by $f(x)$.
For a discussion on the topic of rearrangement of functions one can also see \cite{2}. \\
As it can be seen now in \cite{1} or \cite{3} the following is true
\begin{ctheorem}{A} \label{thm:a}
Let $f\in\mr{BMO}\left([0,1]^n\right)$. Then $f^\star\in\mr{BMO}\left((0,1]\right)$. Moreover there exists a constant $c$ depending only in the dimension of the space such that
\begin{equation} \label{eq:1p3}
\|f^\star\|_\star \leq c\|f\|_\star.
\end{equation}
(For instance a choice for such a constant is $c=2^{n+5}$)
\end{ctheorem}

\noindent Until now it has not been found the best possible value of $c$ in order that \eqref{eq:1p3} holds for any $f\in\mr{BMO}\left([0,1]^n\right)$, for dimensions $n\geq 2$. As for the case $n=1$ the following is true as can be seen in \cite{6}.
\begin{ctheorem}{B}
Let $f\in\mr{BMO}([0,1])$. Then $f^\star, f_d\in\mr{BMO}\left((0,1]\right)$ and the following inequalities hold:
\begin{align}
\|f^\star\|_\star &\leq \|f\|_\star, \label{eq:1p4} \\
\|f_d\|_\star &\leq \|f\|_\star. \label{eq:1p5}
\end{align}
\end{ctheorem}

\noindent Our aim in this paper is to find a better estimation for the constant $c$ that appears in \eqref{eq:1p3}. For this reason we work on the respective dyadic analogue problem. \\
We consider integrable functions defined on $[0,1]^n$ such that the following holds
\begin{equation} \label{eq:1p6}
\|f\|_{\star,\mathcal{D}} \equiv \sup \left\{ \Omega(f,Q) : Q\in \mathcal{D}\right\} < +\infty
\end{equation}
Here by $\mathcal{D}$ we denote the tree of dyadic subcubes of $Q_0\equiv [0,1]^n$, that is the cubes that are produced if we bisect each side of $Q_0$ and continue this process to any resulting cube. Then if \eqref{eq:1p6} holds for $f$, we will say that it belongs to the dyadic $\mr{BMO}$ space, denoted by $\mr{BMO}_\mathcal{D}\left([0,1]^n\right)$. Our first result is the following:
\begin{theorem} \label{thm:1}
Let $f\in\mr{BMO}_\mathcal{D}\left([0,1]^n\right)$. Then $f_d\in\mr{BMO}\left((0,1]\right)$ and
\begin{equation} \label{eq:1p7}
\|f_d\|_\star \leq 2^n\|f\|_{\star,\mathcal{D}}.
\end{equation}
\end{theorem}

\noindent As in the usual case this theorem enables us to prove an inequality of the type of John-Nirenberg (see for example \cite{5}) which can be seen in the following:
\begin{theorem} \label{thm:2}
Let $f\in\mr{BMO}_\mathcal{D}\left([0,1]^n\right)$. Then the following inequality is true
\begin{equation} \label{eq:1p8}
\left|\left\{x\in [0,1]^n: f(x)-f_Q > \lambda\right\}\right| \leq B \exp\left(-\frac{b\lambda}{\|f\|_{\star,\mathcal{D}}}\right),
\end{equation}
for any $\lambda>0$, where $b$ depends only on the dimension of the space, while $B$ is independent of $n$. \big(For example \eqref{eq:1p8} is satisfied for $b = \frac{1}{g^{n-1}\mr e}$ and $B=\mr e$\big).
\end{theorem}

\noindent After proving the above theorems we devote out study to the class of functions that satisfy the dyadic Gurov-Reshetnyak condition. More precisely we consider functions $f: Q_0\equiv [0,1]^n \to \mb R^+$ which are integrable and satisfy
\begin{equation} \label{eq:1p9}
\Omega(f,Q)\leq \varepsilon f_Q,
\end{equation}
for any $Q\in \mathcal{D}$ and some $\varepsilon\in (0,2)$, independent of the cude $Q$. We say then than $f$ satisfies the dyadic Gurov-Reshetnyak condition on $[0,1]^n$ with constant $\varepsilon$ and write it as $f\in \mr{GR}_\mathcal{D}(Q_0,\varepsilon)$. (Note that for any $f\in L^1(Q_0)$, \eqref{eq:1p9} is satisfied for any cube $Q$, for the constant $\varepsilon=2$). The study of such class of functions is of much importance for harmonic analysis and especially in the theory of weights. An extensive presentation of the study of such a class in the non-dyadic case can be seen in \cite{8}. \\
For the study of the class $\mr{GR}_\mathcal{D}(Q_0,\varepsilon)$ we define for any $\ell$ belonging to it the following function
\begin{equation} \label{eq:1p10}
v(f;\sigma) = \sup \left\{\frac{\Omega(f,Q)}{f_Q} : Q\in \mathcal{D},\ \ \text{with}\ \ \ell(Q) \leq \sigma\right\},
\end{equation}
for ant $0\leq \sigma \leq \ell(Q_0)$, where by $\ell(Q)$ we denote the length of the side of the cube $Q$. \\
We will prove the following independent result
\begin{theorem} \label{thm:3}
Let $f\in L^1(Q_0)$ with non-negative values. Then for any $t\in[0,1]$ the following inequality is true:
\begin{equation} \label{eq:1p11}
\frac{1}{t} \int_0^t |f^\star(u) - f^{\star\star}(t)|\,\mr du \leq 2^n f^{\star\star}(t) v(f;\sigma_t),
\end{equation}
where $\sigma_t = \min\left(2t^\frac{1}{n}, 1\right)$.
\end{theorem}

\noindent Here by $f^{\star\star}(t)$ we denote the Hardy function of $f^\star$ defined as $f^{\star\star}(t) = \frac{1}{t} \int_0^t f^\star(u) \mr du$, for $t\in (0,1]$. Moreover we prove the following result by applying Theorem \ref{thm:3}.

\begin{theorem} \label{thm:4}
Let $f\in Q_0\to \mb{R}^+$, $f\in L_1(Q_0)$. Then there exist constants $c_i$ for $i=1, 2, 3, 4$ depending only on $n$ such that the following holds: $c_4>1$ and
\begin{equation} \label{eq:1p12}
f^{\star\star}(t) \leq c_1 f_{Q_0} \exp\left( c_2\int_{c_3 t^\frac{1}{n}}^1 v(f;\sigma)\frac{\mr d\sigma}{\sigma}\right),
\end{equation}
for every $t\in \big(0, \frac{1}{c_4}\big]$.
\end{theorem}

\noindent The proof of Theorem \ref{thm:4} depends on Theorem \ref{thm:3}, and can be effectively used for us to prove the following:
\begin{theorem} \label{thm:5}
Let $f\in\mr{GR}_\mathcal{D}(Q_0,\varepsilon)$ for some $\varepsilon\in\left(0, \frac{1}{2^{n-1}}\right)$. Then for any $t\in(0,1]$ we have that
\begin{equation} \label{eq:1p13}
f^{\star\star}(t) \leq \frac{p}{p-1} f_{Q_0} t^{-\frac 1 p},\ \text{where}\ p>1\ \text{is defined by}\ \frac{p^p}{(p-1)^{p-1}} = \frac{1}{2^{n-1}\varepsilon}
\end{equation}
\end{theorem}

\noindent An immediate consequence is the following
\begin{corollary} \label{cor:1}
Let $f\in\mr{GR}_\mathcal{D}(Q_0,\varepsilon)$ for some $\varepsilon\in\left(0, \frac{1}{2^{n-1}}\right]$. Then $f\in L^q(Q_0)$ for any $q\in [1,p)$, where $p$ is defined by \eqref{eq:1p13}
\end{corollary}

\noindent In this way we increase the integrability properties of $f$, if this belongs to the space $\mr{GR}_\mathcal{D}(Q_0,\varepsilon)$ for a certain range of $\varepsilon$'s.
\medskip

\noindent The paper is originated as follows: \\
In Section \ref{sec:2} we give some preliminaries (Lemmas) needed in subsequent sections. \\
In Section \ref{sec:3} we prove Theorems 1 and 2 and in \ref{sec:4} we provide proofs of Theorems 3,4 and 5.

\section{Preliminaries} \label{sec:2}
Here we state some Lemmas needed in subsequent sections. Those can be found in \cite{8}. The first one is the following:
\begin{lemma} \label{lem:2p1}
Let $f\in L^1(Q_0)$. Then if we define $\Omega(f,Q)$ by \eqref{eq:1p1}, for a certain cube $Q\subseteq Q_0$ we have the following equalities:
\[
\Omega(f,Q) = \frac{2}{|Q|} \int\limits_{\{x\in Q: f(x)>f_Q\}} (f(x)-f_Q)\,\mr dx =
\frac{2}{|Q|} \int\limits_{\{x\in Q: f(x)<f_Q\}} (f_Q - f(x))\,\mr dx.
\]
\end{lemma}

\noindent We will also need the following
\begin{lemma} \label{lem:2p2}
Let $f:I_1\equiv [a_1,b_1]\to\mb R$ be monotone integrable on $I_1$. Suppose we are given $I=[a,b]\subseteq I_1$ such that $f_I=f_{I_1}$. Then the following inequality is true: $\Omega(f,I)\leq \Omega(f,I_1)$.
\end{lemma}

\noindent At last we state
\begin{lemma} \label{lem:2p3}
Let $f$ be non-increasing, summable on $(0,1]$ and let also $F(t) = \frac{1}{t} \int_0^t f(u)\mr du$, for $t\in (0,1]$. Then for ant constant $\gamma >1$ the following inequality is true:
\[
F\left(\frac t \gamma \right) - F(t) \leq \frac \gamma 2 \frac 1 t \int_0^t |f(u) - F(t)|\mr du,\quad t\in(0,1].
\]
\end{lemma}

\section{$f_d$ as an element of $\mr{BMO}\left((0,1]\right)$} \label{sec:3}
Here we suppose that $f$ is defined on $Q_0\equiv [0,1]^n$, is real valued and integrable. We proceed to the (following \cite{6})
\begin{proof}[Proof of theorem \ref{thm:1}]
Suppose that $f\in\mr{BMO}_\mathcal{D}\left([0,1]^n\right)$. We shall prove that $f_d\in\mr{BMO}((0,1])$ and that
\begin{equation} \label{eq:2p1}
\|f_d\|_\star \leq 2^n\|f\|_{\star,\mathcal{D}}.
\end{equation}
For the proof of \eqref{eq:2p1} we need to prove the following inequality
\begin{equation} \label{eq:2p2}
\frac{1}{|J|} \int_J \left|f_d(u)-\left(f_d\right)_J\right|\mr du \leq 2^n \|f\|_{\star,\mathcal{D}},
\end{equation}
for any $J$ interval of $(0,1]$. Fix such a $J$. We set $\alpha = \frac{1}{|J|} \int_J f_d = (f_d)_J$.

\medskip
\noindent \rnum 1) We first consider the case where
\begin{equation} \label{eq:2p3}
\alpha \ge \int_{[0,1]^n} f(x)\mr dx
\end{equation}
\noindent We consider now the family $(D_j)_j$ of those cubes $I\in \mathcal{D}$ maximal with respect to the relation $\subseteq$ under the condition $\frac{1}{|I|} \int_{I} f>\alpha$. Certainly, because of \eqref{eq:2p3} we have that any such cube must be a strict subset of $[0,1]^n$.
Additionally, because of the maximality of every $D_j$ and the tree structure of $\mathcal{D}$ we have that $(D_j)_j$ is a pairwise disjoint subfamily of the tree $\mathcal{D}$. Certainly for any such cube $D_j$ we have that $\frac{1}{|D_j|} \int_{D_j} f>\alpha$, so as a consequence $\frac{1}{|E|} \int_E f>\alpha$, where $E$ denotes the union of the elements of the family $(D_j)_j$, that is $E=\cup_jD_j$.
Now for any dyadic cube $I\neq [0,1]^n$ we denote as $I^\star$ the father of $I$ in $\mathcal{D}$, that is the dyadic cube for which if we bisect it's sides we produce $2^n$ dyadic subcubes of $I^\star$, one of which is $I$. Now we consider for any $D_j$ the respective element of $\mathcal{D}$, $D_j^\star$. \\
We look at the family $(D_j^\star)_j$. Certainly this is not necessarily pairwise disjoint. We consider now a maximal subfamily of $(D_j^\star)_j$, denoted as $(D_{j_k}^\star)_k$, under the relation of $\subseteq$. This is pairwise disjoint and $\cup_k D^\star_{j_k} = \cup_j D_j$.
Moreover for any $k$, $D_{j_k}^\star \supsetneq D_{j_k}$. Additionally by the definition of $E$ and the dyadic version of the Lebesque differentiation theorem, we have that $f(x) \leq \alpha$, for almost every $x\in [0,1]^n \setminus E$. Moreover by the maximality of $D_{j_k}$ we must have that $f_{D_{j_k}^\star} \leq \alpha$, for any $k$.
We now set $E^\star = \cup_k D_{j_k}^\star$. Then $E \subsetneq E^\star$ and $|E| \geq \frac{|E^\star|}{2^n}$, by construction.
We look now upon the function $f_d: (0,1] \to \mb{R}$. Since $\alpha = (f_d)_J > f_{[0,1]^n}$ it is easy to see (since $f_d$ is non-increasing) that there exists $t\in(0,1]$, such that $J\subseteq [0,t]$ and $\frac 1 t \int_0^t f_d(u)\,\mr du = \alpha$. That is $(f_d)_{[0,t]} = (f_d)_J$. \\
We now take advantage of Lemma \ref{lem:2p2}. We obtain immediately that:
\begin{equation} \label{eq:2p4}
\Omega(f_d,J) = \frac{1}{|J|} \int_J \left|f_d(u)-\left( f_d\right)_J\right|\mr du \leq
\Omega(f_d, [0,t]) = \frac{1}{t} \int_0^t |f_d(u)-\alpha|\,\mr du.
\end{equation}
Since now $\frac{1}{|E|} \int_E f > \alpha$ it is immediate because, $f_d$ is non-increasing, that
\[
\frac{1}{|E|} \int_0^{|E|} f_d(u)\,\mr du \geq \frac{1}{|E|} \int_E f(x)\,\mr dx > \alpha = \frac{1}{t} \int_0^t f_d(u)\,\mr du,
\]
so that the measure of $E$ must satisfy $|E|\leq t$.
Thus, since by the above comments mentioned in this proof, we see that $|E^\star| \leq 2^n |E| \leq 2^n t$.

By \eqref{eq:2p4} now, it is enough to prove that
\[
\frac 1 t \int_0^t |f_d(u)-\alpha|\,\mr du \leq 2^n \|f\|_{\star,\mathcal{D}},
\]
for the case \rnum{1}) to be completed. For this reason we proceed as follows: By using Lemma \ref{lem:2p1} we have that
\begin{equation} \label{eq:2p5}
\int\limits_0^t \left|f_d(u)-\alpha\right|\mr du = 2\int_{\left\{u\in(0,t]: f_d(u)>\alpha\right\}} \left(f_d(u)-\alpha\right)\mr du,
\end{equation}
since $\alpha = (f_d)_{(0,t]}$.

The right side now of \eqref{eq:2p5} equals to $2\int_{\{f>\alpha\}}(f(x)-\alpha)\,\mr dx$, because of the equimeasurability of $x\in [0,1]^n$ $f$ and $f_d$ and the fact that $\alpha = \frac{1}{t} \int_0^t f_d(u)\,\mr du \geq f_d(t)$.\\
Thus since $f(x) \leq \alpha$, for almost every element of $[0,1]^n\setminus E^\star$, \eqref{eq:2p5} and the remarks above give that
\begin{multline} \label{eq:2p6}
\int_0^t |f_d(u)-\alpha|\,\mr du = 2\int_{\left\{ x\in E^\star \equiv \cup_k D_{j_k}^\star: f(x)>\alpha\right\}} (f(x)-\alpha)\,\mr dx = \\
2\int_{\left(\cup_k D_{j_k}^\star\right) \cap \left\{f>\alpha\right\}} (f(x)-\alpha)\,\mr dx = 2\sum\limits_k \int_{D_{j_k}^\star \cap \{f>\alpha\}} (f(x)-\alpha)\,\mr dx.
\end{multline}
We prove now that for any $k$, the following inequality holds
\begin{equation} \label{eq:2p7}
\int_{D_{j_k}^\star \cap \{f>\alpha\}} (f(x)-\alpha)\,\mr dx \leq \int_{D_{j_k}^\star \cap \big\{f>f_{D_{j_k}^\star}\big\}}\left(f(x)-f_{D_{j_k}^\star}\right)\mr dx,
\end{equation}
Indeed, \eqref{eq:2p7} is equivalent to
\begin{equation} \label{eq:2p8}
\ell_k \equiv
\int\limits_{D_{j_k}^\star \cap \big\{f_{D_{j_k}^\star} < f \leq \alpha\big\}} f(x)\,\mr dx \geq
f_{D_{j_k}^\star}\left|D_{j_k}^\star \cap \left\{f>f_{D_{j_k}^\star}\right\}\right| - \alpha \left|D_{j_k}^\star \cap \{f>\alpha\}\right|,
\end{equation}
This is now easy to prove since
\begin{align*}
\ell_k & \geq f_{D_{j_k}^\star} \left| D_{j_k}^\star \cap \left\{f_{D_{j_k}^\star}<f\leq \alpha\right\}\right| \\
 & = f_{D_{j_k}^\star} \left| D_{j_k}^\star \cap \left\{f>f_{D_{j_k}^\star}\right\}\right| - f_{D_{j_k}^\star}\left|D_{j_k}^\star \cap \{f>\alpha\}\right| \\
  & \geq f_{D_{j_k}^\star} \left| D_{j_k}^\star \cap \left\{ f>f_{D_{j_k}^\star}\right\}\right|- \alpha \left|D_{j_k}^\star \cap \{f>\alpha\}\right|
\end{align*}
since $f_{D_{j_k}^\star}\leq \alpha$, $\forall k$. \\
But the last inequality is exactly \eqref{eq:2p8}, so by \eqref{eq:2p6} and \eqref{eq:2p7} we have that:
\begin{multline} \label{eq:2p9}
\int_0^t \left|f_d(u)-\alpha\right|\mr du \leq
2\sum_k \int_{D_{j_k}^\star \cap \big\{f>f_{D_{j_k}^\star}\big\}} \left(f(x)-f_{D_{j_k}^\star}\right)\mr dx = \\
\sum_k \int_{D_{j_k}^\star} \left|f(x)-f_{D_{j_k}^\star}\right|\mr dx =
\sum_k \left|D_{j_k}^\star\right|\Omega\left(f,D_{j_k}^\star\right) \leq \\
\left(\sum_k\left|D_{j_k}^\star\right|\right)\|f\|_{\star,\mathcal{D}} =
|E^\star|\,\|f\|_{\star,\mathcal{D}} \leq 2^n t \|f\|_{\star,\mathcal{D}},
\end{multline}
where the first equality in \eqref{eq:2p9} is due to Lemma \ref{lem:2p1}. \\
Thus we have proved that
\[
\frac{1}{t} \int_0^t |f_d(u)-\alpha|\,\mr du \leq 2^n \|f\|_{\star,\mathcal{D}},
\]
and the proof of case \rnum{1}) is complete.

\medskip
We are going now to give a brief discussion for the second case, since this is analogous to the first one.

This is the following: \\
\rnum{2}) We assume that $J$ is a subinterval of $(0,1]$ and that
\begin{equation} \label{eq:2p10}
\alpha = \frac{1}{|J|} \int_J f_d(u)\,\mr du < \int_{[0,1]^n} f(x)\,\mr dx.
\end{equation}
We prove that $\Omega(f_d,J)\leq 2^n\|f\|_{\star,\mathcal{D}}$. \\
By \eqref{eq:2p10} we choose $t\in[0,1)$ such that $\alpha = \frac{1}{t}\int_{1-t}^1 f_d(u)\,\mr du$ and $J\subseteq [1-t,1]$.
We choose the maximal of $(D_j)_j$, $D_j\in \mathcal{D}$ for every $j$ such that $\frac{1}{|D_j|} \int_{D_j} f\leq \alpha$. This is possible in view of \eqref{eq:2p10} and in view of \eqref{eq:2p10} we have that $D_j\neq X$ and because of it's maximality it is pairwise disjoint.
We pass as before to the pairwise disjoint family $(D_{j_k}^\star)_k$, for which we have $E^\star = \cup_k D_{j_k}^\star = \cup_j D_j^\star \supseteq \cup D_j = E$, $|E^\star| \leq 2^n |E|$ and that $h(x)\geq \alpha$, for almost every $x\in [0,1]^n\setminus E$. As before we have
\begin{multline} \label{eq:2p11}
\Omega(f_d,J) =
\frac{1}{|J|} \int_J \left|f_d(u) - \left(f_d\right)_J\right|\mr du \leq
\frac{1}{t} \int_{1-t}^1 |f_d(u)-\alpha|\,\mr du = \\
\frac{2}{t}\int_{\left\{u\in [1-t,1]: f_d(u)<\alpha\right\}} \left(\alpha-f_d(u)\right)\mr du =
\frac{2}{t}\int_{\left\{x\in Q_0\equiv [0,1]^n: f(x)<\alpha\right\}} \left(\alpha-f(x)\right)\mr dx = \\
\frac{2}{t} \int_{\left\{x\in E^\star: f(x)<\alpha\right\}} \left(\alpha-f(x)\right)\mr dx =
\frac{2}{t} \sum_k \int_{D_{j_k}^\star \cap \{f<\alpha\}} \left(\alpha-f(x)\right)\mr dx.
\end{multline}

By the fact that $f_{D_{j_k}^\star} \geq \alpha$ and the same reasoning as before we conclude that for any $k$:
\begin{equation} \label{eq:2p12}
\int_{D_{j_k}^\star\cap \{f<\alpha\}} (\alpha-f(x))\,\mr dx \leq \int_{D_{j_k}^\star \cap \big\{f<f_{D_{j_k}^\star}\big\}} \left( f_{D_{j_k}^\star} - f(x) \right)\mr dx.
\end{equation}
Thus by \eqref{eq:2p11} and \eqref{eq:2p12} we obtain:
\begin{multline} \label{eq:2p13}
\Omega(f_d,J) \leq
\frac{2}{t} \sum_k \int_{D_{j_k}^\star \cap \big\{ f<f_{D_{j_k}^\star}\big\}} \left( f_{D_{j_k}^\star}-f(x)\right)\mr dx = \\
\frac{1}{t} \sum_k |D_{j_k}^\star| \Omega\left(f,D_{j_k}^\star\right) \leq
\frac{|E^\star|}{t} \|f\|_{\star,\mathcal{D}} \leq
\frac{2^n |E|}{t} \|f\|_{\star,\mathcal{D}}.
\end{multline}
As before we can prove that $|E| \leq t$, so then \eqref{eq:2p13} gives the result needed to prove. \\
Thus we proved that for any $J\subseteq [0,1]$ we have $\Omega(f_d,J) \leq 2^n \|f\|_{\star,\mathcal{D}}$, or that $\|f_d\|_\star \leq 2^n \|f\|_{\star,\mathcal{D}}$ and our result is now complete.
\end{proof}
We are now able to prove the following
\begin{ctheorem}{3.1} \label{thm:3p1}
Let $f:Q_0\equiv [0,1]^n \to \mr R$ be such that $\int_{Q_0}f = 0$ and that $f\in \mr{BMO}_{\mathcal{D}}\left([0,1]^n\right)$. Then
\[
f_d(t) \leq \frac{\|f\|_{\star,\mathcal{D}}}{b} \ln\left[\frac{B}{t}\right],
\]
for some constants $b,B>0$ depending only in the dimension $n$.
\end{ctheorem}

\begin{proof}
We define $F(t) = \frac{1}{t}\int_0^t f_d(u)\,\mr du$. Then by Lemma \ref{lem:2p3} $F\left(\frac{t}{\alpha}\right) - F(t) \leq \frac{\alpha}{2} \frac{1}{t} \int_0^t \left|f_d(u) - F(t)\right|\mr du$, for any $t\in (0,1]$ and $\alpha>1$. Thus
\begin{equation} \label{eq:2p14}
F\left(\frac{t}{\alpha}\right) - F(t) \leq \frac{\alpha}{2} \Omega\left(f_d,[0,t]\right) \leq \frac{\alpha}{2}\|f_d\|_\star \leq 2^{n-1} \alpha \|f\|_{\star,\mathcal{D}},
\end{equation}
by using Theorem \ref{thm:1}. By \eqref{eq:2p14} now we have for any $\alpha>1$ the following inequalities
\begin{equation} \label{eq:2p15}
F\left(\frac{1}{\alpha^i}\right) - F\left(\frac{1}{\alpha^{i-1}}\right) \leq 2^{n-1} \alpha \|f\|_{\star,\mathcal{D}},
\end{equation}
for any $i=1, 2, \ldots, k, k+1$ and for any fixed $k\in \mb N$.
Summing inequalities \eqref{eq:2p15} we obtain as a consequence that
\begin{multline} \label{eq:2p16}
F\left(\frac{1}{\alpha^{k+1}}\right) - F(1) \leq
(k+1) 2^{n-1} \alpha \|f\|_{\star,\mathcal{D}} \implies
(\text{since}\ \int_{[0,1]} f=0) \\
F\left(\frac{1}{\alpha^{k+1}}\right) \leq
\left((k+1) 2^{n-1} \alpha \right) \|f\|_{\star,\mathcal{D}},
\end{multline}
Fix now $t\in(0,1]$ and $\alpha>1$. Then for a unique $k\in\mb N$ we have that
\begin{multline} \label{eq:2p17}
\frac{1}{\alpha^{k+1}} < t \leq \frac{1}{\alpha^k}\implies
k\leq \frac{1}{\ln(\alpha)} \ln\left(\frac{1}{t}\right) \overset{\eqref{eq:2p16}}{\implies} \\
f_d(t) \leq \frac{1}{t} \int_0^t f_d(u)\,\mr du =
F(t) \leq F\left(\frac{1}{\alpha^{k+1}}\right) \leq
\left((k+1) 2^{n-1} \alpha\right)\|f\|_{\star,\mathcal{D}} \leq \\
\left(\left[\frac{1}{\ln(\alpha)}\ln\left(\frac{1}{t}\right)+1\right] 2^{n-1} \alpha\right) \|f\|_{\star,\mathcal{D}}.
\end{multline}
Now the function $h$ defined for any $\alpha>1$, by $h(\alpha)=\frac{\alpha}{\ln(\alpha)}$ takes it's minimum value at $\alpha=\mr e$. Thus for this value of $\alpha$, we obtain by \eqref{eq:2p17}
\[
f_d(t) \leq \left[\ln\left(\frac{1}{t}\right) 2^{n-1} \mr e + 2^{n-1} \mr e\right] \|f\|_{\star,\mathcal{D}} =
\frac{\|f\|_{\star,\mathcal{D}}}{b} \left[\ln\left(\frac{B}{t}\right)\right],\ \text{for any}\ t\in(0,1]\]
where $b=\frac{1}{2^{n-1}\mr e}$, $B=\mr e$.
\end{proof}

\medskip
\noindent We now proceed to
\begin{proof}[Proof of Theorem \ref{thm:2}]
For any $f\in\mr{BMO}_\mathcal{D}\left([0,1]^n\right)$ we prove that
\begin{equation} \label{eq:2p18}
\left|\left\{x\in Q_0: \left(f(x)-f_{Q_0}\right)>\lambda\right\}\right| \leq B \exp\left(-\frac{b\lambda}{\|f\|_{\star,\mathcal{D}}}\right),
\end{equation}
for every $\lambda>0$ and the above values of $b,B$. \\
We fix a $\lambda > 0$ and suppose without loss of generality that $f_{Q_0}=0$. We set $A_\lambda = \left\{x\in Q_0: f(x)>\lambda\right\}$. In order to prove \eqref{eq:2p18}, we just need to prove that $|A_\lambda| \leq B\exp\left(-\frac{b\lambda}{\|f\|_{\star,\mathcal{D}}}\right)$, for this value of $\lambda>0$.
We have $|A_\lambda| = |\{f>\lambda\}| = |\{f_d>\lambda\}| \leq \left|\left\{ t\in (0,1]: \frac{\|f\|_{\star,\mathcal{D}}}{b}\ln\left(\frac{B}{t}\right) > \lambda\right\}\right|$, since by Theorem \ref{thm:3p1} $f_d(t) \leq \frac{\|f\|_{\star,\mathcal{D}}}{b} \ln\left(\frac{B}{t}\right)$, for every $t\in(0,1]$. Thus we have that
\[
|A_\lambda| \leq
\left|\left\{ t\in(0,1]: t < \exp\left(-\frac{b\lambda}{\|f\|_{\star,\mathcal{D}}}\right)\right\}\right| \leq
B\exp\left(-\frac{b\lambda}{\|f\|_{\star,\mathcal{D}}}\right)
\]
\end{proof}

\begin{remark} \label{rem:3p1}
By considering the results of this section it is worth mentioning the following. Suppose that $f: [0,1]^n\to \mr R^+$ be such that $\|f\|_{\star,\mathcal{D}}< +\infty$. Because $f$ is non-negative we must have that $f_d = |f|_d = f^\star$, on $(0,1]$.
Thus we have that for any such $f$ we must have that $\|f^\star\|_\star \leq 2^n \|f\|_{\star,\mathcal{D}}$ and the inequality $\left|\left\{ x\in Q_0:\left|f(x)-f_{Q_0}\right| > \lambda \right\}\right| \leq B \exp\left(-\frac{b\lambda}{\|f\|_{\star,\mathcal{D}}}\right)$, for every $\lambda>0$ and the above mentioned values of $b$ and $B$.
\end{remark}

\section{The dyadic Gurov-Reshetnyak condition} \label{sec:4}
We again consider functions $f: Q_0\equiv [0,1]^n \to \mb R^+$ such that $f\in L^1(Q_0)$ and the following condition is satisfied
\[
\Omega(f,Q) \equiv \frac{1}{|Q|} \int_Q |f(x) - f_Q|\,\mr dx \leq \varepsilon f_Q,\quad \forall Q\in \mathcal{D},
\]
for some $\varepsilon\in (0,2)$, independent of the cube $Q$. As we noted in Section \ref{sec:1} we say then that $f\in \mr{GR}_\mathcal{D}(Q_0,\varepsilon)$.
Define the function $v(f; \cdot)$ by \eqref{eq:1p10}.\\
We are going to give the
\begin{proof}[Proof of Theorem \ref{thm:3}]
We define $\sigma_t = \min\left(2 t^{\frac{1}{n}}, 1\right)$, for every $t\in (0,1]$ and $B_t = v(f; \sigma_t)$. We shall prove that for an $t\in (0,1]$ we have that
\[
\frac{1}{t} \int_0^t |f^\star(u) - f^{\star\star}(t)|\,\mr du \leq 2^n B_t f^{\star\star}(t).
\]
Fix a $t\in (0,1]$ and set $\alpha = f^{\star\star}(t)$. Then $\alpha > f_{Q_0} = \int_{[0,1]^n} f(x)\,\mr dx = f^{\star\star}(1)$, since $f^\star$ is non-increasing we define now the following operator
\[
M_d\varphi(x) = \sup\left\{ \frac{1}{|Q|} \int_Q |\varphi(y)|\,\mr dy: x\in Q\in \mathcal{D}\right\}
\]
for every $\varphi\in L^1(Q_0)$, where $\mathcal{D}$ is as usual the class of all dyadic subcubes of $Q_0$. This is called the dyadic maximal operator with respect to the tree $\mathcal{D}$. \\
We consider the set $E = \left\{M_df > \alpha\right\}$. For any $x\in E$, there exists $Q_x\in \mathcal{D}: x\in Q_x$ and $\frac{1}{|Q_x|} \int_{Q_x} f > \alpha$.
Consider for every such $x$ the collection of all $Q_x\in \mathcal{D}$ with the above property and choose the one with maximal measure. Note that each two sets of the above collection have the property that one of them contains the other, because of the tree structure of $\mathcal{D}$.

From the above remarks we have that $E$ can be written as $E = \cup_j D_j$, where $(D_j)_j$ is a pairwise disjoint family of cubes in $\mathcal{D}$, maximal under the condition $\frac{1}{|D_j|} \int_{D_j} f > \alpha$.
Since $\alpha > f_{Q_0} = \int_{[0,1]^n} f$ for any such cube we have that $D_j \neq [0,1]^n$. Let also $D_j^\star$ be the father of $D_j$ in $\mathcal{D}$, for every $j$.
By the maximality of $D_j$ we have that $\frac{1}{|D_j^\star|} \int_{D_j^\star} f \leq \alpha$, or that $f_{D_j^\star} \leq \alpha$. We set now $E^\star = \cup_j D_j^\star$. \\
Then $E^\star$ can be written as $E^\star = \cup_k D_{j_k}^\star$, where the family $\left(D_{j_k}^\star\right)_k$ is a maximal subfamily of $(D_j)_j$ under the relation $\subseteq$. Because of its maximality, this must be disjoint.
Of course by the dyadic form of the Lebesque differentiation theorem we have that for almost every $x\notin E$, $x\in [0,1]^n$, the following is satisfied:
\begin{equation} \label{eq:4p1}
f(x) \leq M_df(x) \leq \alpha = f^{\star\star}(t).
\end{equation}
We consider now the following quantity $L_t \equiv \int_0^t |f^\star(u) - f^{\star\star}(t)|\,\mr du$, which in view of Lemma \ref{lem:2p1} can be written as
\begin{equation} \label{eq:4p2}
L_t = 2 \int_{\left\{ u\in(0,t]: f^\star(u) \geq \alpha \right\}} (f^\star(u) - \alpha)\,\mr du,
\end{equation}
By \eqref{eq:4p2} we have that
\begin{equation} \label{eq:4p3}
L_t = 2\int_{\left\{ x\in[0,1]^n: f(x) > \alpha \right\}} (f(x)-\alpha)\,\mr dx,
\end{equation}
because of the equimeasurability of $f$ and $f^\star$ and the fact that $\alpha = f^{\star\star}(t) = \frac{1}{t} \int_0^t f^\star(u)\,\mr du \geq \int_0^1 f^\star$. \\
Then since $E \subseteq E^\star$ and because of \eqref{eq:4p1} we have as a consequence from \eqref{eq:4p3} that
\begin{multline} \label{eq:4p4}
L_t = 2\int_{E^\star \cap \left\{ x\in Q_0: f(x)>\alpha\right\}} (f(x)-\alpha)\,\mr dx =
2\int _{\left(\cup D_{j_k}^\star\right) \cap \{f>\alpha\}} (f(x)-\alpha)\,\mr dx = \\
\sum_k \int_{D_{j_k}^\star \cap \{f>\alpha\}} (f(x)-\alpha)\,\mr dx.
\end{multline}
Is is now easy to show, as in Section \ref{sec:3} that the following inequality is true
\begin{equation} \label{eq:4p5}
\int_{D_{j_k}^\star \cap \{f>\alpha\}} (f(x)-\alpha)\,\mr dx \leq \int_{D_{j_k}^\star \cap \big\{f>f_{D_{j_k}^\star}\big\}} \left(f(x)-f_{D_{j_k}^\star}\right)\mr dx,\ \text{for any}\ k.
\end{equation}
\eqref{eq:4p4} now, in view of \eqref{eq:4p5} becomes:
\begin{multline*}
L_t \leq \sum_k 2\int_{\big\{x\in D_{j_k}^\star: f(x)>f_{D_{j_k}^\star}\big\}} \left(f(x)-f_{D_{j_k}^\star}\right)\mr dx = \\
\sum_k \int_{D_{j_k}^\star} \left|f(x)-f_{D_{j_k}^\star}\right|\mr dx =
\sum \left|D_{j_k}^\star\right| \Omega\left(f, D_{j_k}^\star\right),
\end{multline*}
where the first equality holds because of Lemma \ref{lem:2p1}. \\
Now  by the definition of $E$ and $\alpha$ we immediatelly have that
\[
\frac{1}{|E|} \int_0^{|E|} f^\star(u)\,\mr du \geq
\frac{1}{|E|} \int_E f(x)\,\mr dx > \alpha =
\frac{1}{t} \int_0^t f^\star(u)\,\mr du \implies
|E| \leq t,
\]
since $f^\star$ is non-increasing.
Thus by the construction of $E^\star$ we have that
\[
|E^\star| = \sum_k |D_{j_k}^\star| \leq
\sum_k 2^n \left|D_{j_k}^\star \cap E\right| =
2^n |E| \leq 2^n t.
\]

Additionally, for any $k$ we have that $|D_{j_k}^\star| \leq |E^\star| \leq 2^n t$, thus $\ell(D_{j_k}^\star)\leq 2 t^\frac{1}{n}$, for every $k$. \\
Thus we immediately have by the definition of the function $v(f;\cdot)$, that $\Omega(f,D_{j_k}^\star) \leq v(f;\sigma_t) f_{D_{j_k}^\star} \leq v(f;\sigma_t) \alpha = v(f;\sigma_t) f^{\star\star}(t)$, where $\sigma_t = \min\big\{ 2t^\frac{1}{n}, 1\big\}$, for any $t\in (0,1]$. So as a consequence from \eqref{eq:2p8} we obtain $L_t \leq |E^\star| v(f;\sigma_t) f^{\star\star}(t) \leq 2^n t f^{\star\star}(t) B_t \implies \frac{1}{t} \int_0^t |f^\star(u) - f^{\star\star}(t)|\,\mr du \le 2^n B_t f^{\star\star}(t)$.

Thus the proof of our Theorem is complete.
\end{proof}

\noindent We proceed now to the
\begin{proof}[Proof of theorem \ref{thm:4}]
We suppose that we are given $f: Q_0\equiv[0,1]^n \to \mb R^+$ such that $f\in L^1(Q_0)$. By Lemma \ref{lem:2p3} we have that
\begin{equation} \label{eq:4p6}
f^{\star\star}\left(\frac{t}{\gamma}\right) - f^\star(t) \leq
\frac{\gamma}{2} \frac{1}{t} \int_0^t |f^\star(u)-f^{\star\star}(t)|\,\mr du,
\end{equation}
for $t\in (0,1]$ and any $\gamma>1$. \\
Let $t\in (0,1]$, because of Theorem \ref{thm:3} and \eqref{eq:4p6} we have that
\begin{equation} \label{eq:4p7}
f^{\star\star}\left(\frac{t}{\gamma}\right) - f^{\star\star}(t) \leq 2^{n-1}\gamma B_t f^{\star\star}(t) \Rightarrow
f^{\star\star}\left(\frac{t}{\gamma}\right) \leq \left(1 + 2^{n-1} \gamma B_t\right) f^{\star\star}(t),
\end{equation}
We consider now those $t$ for which $t\in\left(0, \frac{1}{2^n \gamma}\right]$. The choice of $\gamma$ will be made later. \\
We set $s = \left[\frac{\ln\left(\frac{1}{2^n t}\right)}{\ln(\gamma)}\right] \in \mb N^\star$. Then we have that $\gamma^s \leq \frac{1}{2^n t} < \gamma^{s+1} \implies \gamma^s t > \frac{1}{2^n \gamma}$. As a consequence we produce
\begin{equation} \label{eq:4p8}
f^{\star\star}(\gamma^s t)\leq
f^{\star\star}\left(\frac{1}{2^n \gamma}\right) =
2^n \gamma \int_0^{\frac{1}{2^n\gamma}} f^\star \leq
2^n \gamma \int_0^1 f^\star =
2^n \gamma f_{Q_0}.
\end{equation}

Now in view of \eqref{eq:4p7} we must have that
\begin{multline} \label{eq:4p9}
f^{\star\star}\left(\frac{t}{\gamma}\right) \leq
\left(1 + 2^{n-1}\gamma B_t\right) f^{\star\star}(t) \leq
\left(1 + 2^{n-1}\gamma B_t\right) \left(1 + 2^{n-1}\gamma B_{\gamma t}\right) f^\star(2t) \leq \\
\ldots \leq
\prod_{i=0}^s \left(1 + 2^{n-1}\gamma B_{\gamma^it}\right) f^{\star\star}(\gamma^st),
\end{multline}
where $s$ is as above. So that \eqref{eq:4p8} and \eqref{eq:4p9} give
\begin{equation} \label{eq:4p10}
f^{\star\star}\left(\frac{t}{\gamma}\right) \leq 2^n \gamma f_{Q_0} \exp\left(2^{n-1} \gamma \sum_{i=0}^s B_{\gamma^it}\right).
\end{equation}
In view of the inequality $1+x \leq \mr e^x$, which holds for every $x>0$. By the choice of $s$ we have that $(\gamma^i t) 2^n \leq 1$. \\
Thus by the definition of the function $t \longmapsto B_t$ we have
\begin{equation} \label{eq:4p11}
B_{\gamma^i t} = v\left(f; 2(\gamma^i t)^\frac{1}{n}\right),\quad \text{for every}\ \ i=0, 1, 2, \ldots, s.
\end{equation}
Thus
\begin{multline} \label{eq:4p12}
\ell_{k,n} = \int_{2(\gamma^k t)^\frac{1}{n}}^{2(\gamma^{k+1}t)^\frac{1}{n}} v(f;\sigma)\,\frac{\mr d\sigma}{\sigma} \geq \\
v\left(f;2(\gamma^k t)^\frac{1}{n}\right) \left\{ \ln\left[2(\gamma^{k+1}t)^\frac{1}{n}\right] - \ln\left[2(\gamma^k t)^\frac{1}{n}\right]\right\},
\end{multline}
for every $k\in\{0, 1, 2,\ldots, s\}$,
in view of the fact that the function $\sigma \longmapsto v(f;\sigma)$ is non-decreasing. \\
We immediately get from \eqref{eq:4p12} that
\begin{equation} \label{eq:4p13}
\int_{2(\gamma^kt)^\frac{1}{n}}^{2(\gamma^{k+1}t)^\frac{1}{n}} v(f;\sigma)\,\frac{\mr d\sigma}{\sigma} \geq
v\left(f; 2(\gamma^kt)^\frac{1}{n}\right) \ln\left(\gamma^\frac{1}{n}\right),
\end{equation}
>From \eqref{eq:4p10}, \eqref{eq:4p11} and \eqref{eq:4p13} we see that:
\begin{multline} \label{eq:4p14}
f^{\star\star}\left(\frac{t}{\gamma}\right) \leq \\
2^n \gamma f_{Q_0} \exp\left[ 2^{n-1} \gamma \sum_{k=0}^{s-1} \frac{n}{\ln(\gamma)} \int_{2(\gamma^kt)^\frac{1}{n}}^{2(\gamma^{k+1}t)^\frac{1}{n}} v(f;\sigma)\,\frac{\mr d\sigma}{\sigma} + 2^{n-1} \gamma B_{\gamma^st} \right],
\end{multline}
>From \eqref{eq:4p14} we have as a consequence that
\begin{equation} \label{eq:4p15}
f^{\star\star}\left(\frac{t}{\gamma}\right) \leq
2^n \gamma f_{Q_0} \exp\left[ 2^{n-1} n \frac{\gamma}{\ln(\gamma)} \int_{2t^\frac{1}{n}}^1 v(f;\sigma)\,\frac{\mr d\sigma}{\sigma} + 2^{n-1} \gamma v\left(f; \frac{1}{2^n}\right) \right],
\end{equation}
and this holds for every $t\in \big(0, \frac{1}{2^n\gamma}\big]$ and any $\gamma > 1$.
We choose now in \eqref{eq:4p15} $\gamma=\mr e$ in order that the function $\gamma \longmapsto \frac{\gamma}{\ln(\gamma)}$, is minimized on $(1,+\infty)$. Then
\begin{multline} \label{eq:4p16}
\eqref{eq:4p15} \implies f^{\star\star}\left(\frac{t}{\mr e}\right) \leq \\
2^n \mr e f_{Q_0} \exp\left[ 2^{n-1} \mr e n \int_{2t^\frac{1}{n}}^1 v(f;\sigma)\,\frac{\mr d\sigma}{\sigma} \right] \exp\left[ 2^{n-1} \mr e v\left(f;\frac{1}{2^n}\right) \right],
\end{multline}
for every $t\in \big(0, \frac{1}{2^n\mr e}\big]$. Certainly $v\left(f;\frac{1}{2^n}\right)\leq 2$. Thus \eqref{eq:4p16} gives
\begin{equation} \label{eq:4p17}
f^{\star\star}\left(\frac{t}{\mr e}\right) \leq
C_1 f_{Q_0} \exp\left( C_2 \int_{C_3' t^\frac{1}{n}}^1 v(f;\sigma)\,\frac{\mr d\sigma}{\sigma}\right),
\end{equation}
for every $t\in\big(0, \frac{1}{2^n\gamma \mr e}\big]$, for certain constants $C_1, C_2, C_3'$.
By setting $y = \frac{t}{\mr e}$ in \eqref{eq:4p17}, we conclude that for every $y\in \big(0,\frac{1}{2^n\mr e^2}\big]$ the following inequality holds:
\begin{equation} \label{eq:4p18}
f^{\star\star}(y) \leq c_1 f_{Q_0} \exp\left( c_2 \int_{C_3'\mr e^\frac{1}{n}y^\frac{1}{n}} v(f;\sigma)\,\frac{\mr d\sigma}{\sigma} \right),
\end{equation}
where
\begin{align*}
c_1 &= 2^n \mr e \exp[2^n\mr e] = 2^n \exp[2^n \mr e + 1] \\
c_2 &= 2^{n-1} \mr e n \\
\text{and}\ \ c_3 &= C_3' \mr e^\frac{1}{n}.
\end{align*}
So by setting $c_4 = \frac{1}{2^n \mr e^2}$, we derive the proof of our Theorem.
\end{proof}

\medskip
\noindent We are now ready to give the
\begin{proof}[Proof of Theorem \ref{thm:5}]
We are given a function $f\in \mr{GR}_\mathcal{D}(Q_0,\varepsilon)$ for some $\varepsilon: 0< \varepsilon< \frac{1}{2^{n-1}}$ and suppose that $t\in(0,1]$ is fixed. By using Theorem \ref{thm:3} we obtain:
\begin{equation} \label{eq:4p19}
\frac{1}{t} \int_0^t |f^\star(u)-f^{\star\star}(t)|\,\mr du \leq
2^n \varepsilon f^{\star\star}(t).
\end{equation}
Then by Lemma \ref{lem:2p3} we have in view of \eqref{eq:4p13} that
\[
f^{\star\star}\left(\frac{t}{\gamma}\right) \leq \left( 2^{n-1} \gamma \varepsilon + 1\right) f^{\star\star}(t),\ \text{for any}\ \gamma > 1.
\]
Let now $p_0$ be the unique $p>1$ such that $\frac{p^p}{(p-1)^{p-1}} = \frac{1}{2^{n-1}\varepsilon}$. We set $\gamma = \left(\frac{p_0}{p_0-1}\right)^{p_0} \implies \gamma^\frac{1}{p_0} = \frac{p_0}{p_0-1}$.
Then $\left(2^{n-1} \gamma \varepsilon + 1\right)^{p_0} = \left(2^{n-1} \varepsilon \left(\frac{p_0}{p_0-1}\right)^{p_0} + 1\right)^{p_0} = \left(2^{n-1} \varepsilon \frac{1}{p_0-1} \frac{1}{2^{n-1}\varepsilon} + 1\right)^{p_0} = \left(1 + \frac{1}{p_0-1}\right)^{p_0} = \left(\frac{p_0}{p_0-1}\right)^{p_0} = \gamma \implies \left(2^{n-1} \gamma \varepsilon + 1\right) = \gamma^\frac{1}{p_0}$, for a certain $\gamma > 1$ given as above. \\
Thus
\begin{equation} \label{eq:4p20}
f^{\star\star}\left(\frac{t}{\gamma}\right) \leq \gamma^\frac{1}{p_0} f^{\star\star}(t),\quad \forall t\in(0,1].
\end{equation}
Let now $j\in\mb N$ be such that
\begin{equation} \label{eq:4p21}
\gamma^{-j}< t\leq \gamma^{-j+1},
\end{equation}
then by \eqref{eq:4p20} we see inductively that $f^{\star\star}(\gamma^{-k}) \leq \gamma^\frac{k}{p_0} f^{\star\star}(1)$, for any $k\in\mb N$, so by using \eqref{eq:4p21} for our $t$ we conclude that
\begin{equation} \label{eq:4p22}
f^{\star\star}(t) \leq f^{\star\star}(\gamma^{-j}) \leq \gamma^\frac{j}{p_0} f^{\star\star}(1).
\end{equation}
By \eqref{eq:4p21} now $\gamma^\frac{j}{p_0} \leq \left(\frac{\gamma}{t}\right)^\frac{1}{p_0}$. Thus from this last inequality and \eqref{eq:4p22} we have that:
\[
f^{\star\star}(t) \leq \frac{\gamma^\frac{1}{p_0}}{t^\frac{1}{p_0}} f^{\star\star}(1) = \left(\frac{p_0}{p_0-1}\right) f_{Q_0} t^{-\frac{1}{p_0}},
\]
and this holds for any $t\in (0,1]$. The proof of Theorem \ref{thm:5} is now complete.
\end{proof}

At last we mention that the proof of Corollary \ref{cor:1}, is immediate by the statement of Theorem \ref{thm:5}.

\vspace{50pt}
\noindent Eleftherios N. Nikolidakis \\
Post-doctoral researcher\\
National and Kapodistrian University of Athens\\
Department of Mathematics\\
Panepisimioupolis, Zografou 157-84, Athens, Greece\\
E-mail address: lefteris@math.uoc.gr

\end{document}